\documentclass[11pt]{article}

\usepackage[english]{babel}
\usepackage[pass]{geometry}
\usepackage{amsmath,amsthm,amssymb,latexsym}
\usepackage{authblk}

\usepackage[svgnames]{xcolor}
\usepackage[colorlinks=true,linkcolor=MidnightBlue,citecolor=MidnightBlue,urlcolor=MidnightBlue,pdfpagelabels=false]{hyperref}

\usepackage{mathtools}

\usepackage[capitalize]{cleveref}
\usepackage{todonotes}

\usepackage{thmtools}
\theoremstyle{plain}
  \declaretheorem[numberwithin=section]{theorem}
  \declaretheorem[numberlike=theorem]{corollary}
  \declaretheorem[numberlike=theorem]{proposition}
  \declaretheorem[numberlike=theorem]{lemma}
  \declaretheorem[numberlike=theorem]{conjecture}
  \declaretheorem[numberlike=theorem]{question}
\theoremstyle{definition}
  \declaretheorem[numberlike=theorem]{definition}
  \declaretheorem[numberlike=theorem]{example}
  \declaretheorem[numberlike=theorem]{remark}

\newcommand{\ps}[1]{[\![#1]\!]}
\newcommand{\ZZ}{\mathbb{Z}}
\newcommand{\QQ}{\mathbb{Q}}

\newcommand{\diag}[0]{\mathrm{Diag}}
\newcommand{\ct}{\operatorname{ct}}

\newcommand{\nequiv}{\mathrel{\not\equiv}}

\begin{document}

\title{On the representability of sequences\\ as constant terms}

\author[1]{\href{https://specfun.inria.fr/bostan/}{Alin Bostan}}
\author[2]{\href{https://arminstraub.com}{Armin Straub}}
\author[1,3]{\href{https://homepage.univie.ac.at/sergey.yurkevich/}{Sergey Yurkevich}}
\affil[1]{Inria, Univ. Paris-Saclay, France}
\affil[ ]{\emph{\texttt{\href{mailto:alin.bostan@inria.fr}{alin.bostan@inria.fr}}}}
\affil[2]{University of South Alabama, USA}
\affil[ ]{\emph{\texttt{\href{mailto:straub@southalabama.edu}{straub@southalabama.edu}}}}
\affil[3]{University of Vienna, Austria}
\affil[ ]{\emph{\texttt{\href{mailto:sergey.yurkevich@univie.ac.at}{sergey.yurkevich@univie.ac.at}}}}
\date{}                   
\setcounter{Maxaffil}{0}
\renewcommand\Affilfont{\itshape\small}

\maketitle

\begin{abstract}
A constant term sequence is a sequence of rational numbers whose $n$-th
term is the constant term of $P^n(\boldsymbol{x}) Q(\boldsymbol{x})$,
where $P(\boldsymbol{x})$ and $Q(\boldsymbol{x})$ are multivariate
Laurent polynomials.
While the generating functions of such sequences are invariably
diagonals of multivariate rational functions, and hence special period
functions, it is a famous open question, raised by Don Zagier, to
classify diagonals that are constant terms. In this paper, we provide
such a classification in the case of sequences satisfying linear
recurrences with constant coefficients.
We also consider the case of hypergeometric sequences and, for a simple
illustrative family of hypergeometric sequences, classify those that are
constant terms.
\end{abstract}


\section{Introduction}

Recognizing and interpreting integrality of sequences defined by
recursions is at the same time an extensively studied and a difficult
topic in number theory. Even in the case of \emph{P-finite} sequences
$A(n)$ (also called \emph{P-recursive}, or \emph{holonomic}), defined by
linear recurrences with polynomial coefficients
\[
p_r(n) A(n+r) = p_{r-1}(n) A(n+r-1) + \cdots + p_0(n) A(n), \quad p_i(n) \in \ZZ[n],
\]
neither a criterion nor even an algorithm is known for classifying/deciding integrality. An attempt for such a classification is the famous and widely open conjecture by Christol~\cite[Conjecture 4, p.~55]{christol-glob}. Roughly speaking, it states that a P-finite sequence $(A(n))_{n\geq0}$ with (at most) geometric growth is integral if and only if $(A(n))_{n\geq0}$ is the coefficient sequence of the diagonal of a rational function $R(\boldsymbol{x}) \in \ZZ(x_1,\dots,x_d) \cap \ZZ\ps{x_1,\dots,x_d}$ for some $d\geq 1$. Recall that the diagonal of a multivariate power series
\begin{equation} \label{eq:diag}
  R(\boldsymbol{x}) = \sum_{n_1, n_2, \dots, n_d \geq 0} c (n_1, n_2, \dots, n_d)
   x_1^{n_1}  x_2^{n_2}\cdots x_d^{n_d}
\end{equation}
is the univariate power series $\diag(R)$ whose coefficient sequence is given by $A (n) = c (n, n, \ldots, n)$. For a precise statement of Christol's conjecture see \cref{conj:christol} below.

Often integrality of sequences can be explained by the underlying combinatorial nature. For example, the Catalan numbers $C(n)$ satisfying
\[
(n+2) C(n+1) = 2(2n+1) C(n), \quad C(0)=1,
\]
are clearly integers because they count triangulations of convex polygons with $n+2$ vertices. On the other hand, for many other integral and P-finite sequences, combinatorial interpretations are not \textit{a priori} known; this is the case, for instance, for the Apéry numbers $A(n)$ (associated with the irrationality proof of $\zeta(3)$) defined by 
\begin{align*}
(n+1)^3 A(n+1) = (2n + 1)(17n^2 + 17n + 5) A(n) & - n^3 A(n-1), \\ & A(0) = 1, A(1) = 5.
\end{align*}
In both examples above, integrality can be seen from the explicit formulas
\[
C(n) = \binom{2n}{n} - \binom{2n}{n+1} \quad \text{and} \quad A(n) = \sum_{k=0}^n \binom{n}{k}^2\binom{n+k}{k}^2.
\]
Putting Christol's conjecture in practice gives a different justification for the integrality of these two examples. It namely holds that
\begin{align*}
\sum_{n \geq 0} C(n) t^n &= \diag \left(\frac{1-y}{1- x(y+1)^2} \right) 
\quad \text{and}\\
\sum_{n \geq 0} A(n) t^n &= \diag \left( \frac{1}{1-(xy+x+y)(zw+z+w)} \right),
\end{align*}
and the integrality of $C(n)$ and $A(n)$ follows from that of the coefficients in the Taylor expansions of the corresponding multivariate rational functions.

In the context of the current text, however, we would like to emphasize a slightly different viewpoint, which does not only justify integrality of the two examples, but also implies some interesting arithmetic properties.  Writing $\ct[P(\boldsymbol{x})]$ for the constant term of a Laurent polynomial $P(\boldsymbol{x}) \in \QQ[x_1^{\pm1},\dots,x_d^{\pm1}]$, one can prove that \cite[Rem.~1.4]{Straub14}
\begin{align*}
    C(n) & = \ct\left[ (x^{-1} + 2 + x)^n(1-x) \right] \quad \text{and}\\
    A(n) & = \ct\left[ \left(\frac{(x + y)(z + 1)(x + y + z)(y + x + 1)}{xyz}\right)^n \right].
\end{align*}

Similar identities as in the examples of the Catalan and Apéry numbers can be deduced for many other integral P-finite sequences. This motivates the following definition and the subsequent natural question.
\begin{definition}
A sequence $A(n)$ is a \emph{constant term} if it can be represented as
\begin{equation}
  A (n) = \ct [P (\boldsymbol{x})^n Q (\boldsymbol{x})], \label{eq:ct:pq}
\end{equation}
where $P, Q \in \mathbb{Q} [\boldsymbol{x}^{\pm 1}]$ are Laurent polynomials in
$\boldsymbol{x}= (x_1, \ldots, x_d)$. 
\end{definition}
Using the geometric series it is easy to see that generating functions of constant term sequences can be expressed as diagonals of rational functions. The converse is, however, not true in general. This leads to the following question which was raised by Zagier \cite[p.~769, Question~2]{zagier-de} and Gorodetsky \cite{gorodetsky-ct} in the case $Q = 1$ (see \cref{lem:minton:ct} below for an indication why this case is of
particular arithmetic significance). 
    \begin{question}	 \label{Q1}
    \textit{Which P-finite sequences are constant terms?}
    \end{question}

To our knowledge, \cref{Q1} is widely open. In fact, the initial motivation for the present text was the goal of answering the following very particular sub-question asked by the second author in~\cite[Question 5.1]{s-schemes}: 

    \begin{question}	 \label{Q2}
    \textit{Is the Fibonacci sequence $(F(n))_{n\geq0}$ a constant term sequence?}
    \end{question}
Recall that the Fibonacci sequence is the coefficient sequence in the Taylor expansion of the univariate rational function $x/(1-x-x^2)$, or equivalently the P-finite sequence $(F(n))_{n\geq0}$ defined by $F(n+2) = F(n+1) + F(n)$ and $F(0)=0, F(1)=1$.

Already in~\cite{s-schemes} the second author noted that a representation of the Fibonacci numbers as constant terms with $Q = 1$ is impossible since $(F(n))_{n\geq0}$ does not satisfy the so-called \emph{Gauss congruences} (see~\eqref{eq:gauss}). Exploiting the fact that for any prime $p$, the value $F(p) \pmod{p}$ depends on $p \pmod{5}$, we can show (see~\cref{ex:fibonacci}) that the answer to \cref{Q2} is negative. The reason for this is that, as we will prove, for any constant term sequence $A(n)$, the sequence $A(p) \pmod{p}$ must be constant for large enough primes~$p$. Note that this is not a sufficient criterion, since already the Lucas numbers $L(n)$ (defined by the same recursion as the Fibonacci numbers, but with different initial terms $L(0)=2, L(1)=1$, see~\eqref{eq:lucas:L}) do satisfy the Gauss congruences but {are not} constant terms (see \cref{ex:lucas}). 

In the present text, we are able to answer \cref{Q1} in the case of diagonals of rational functions
$F (x) \in \mathbb{Q} (x)$ in a single variable. Such sequences are precisely the (rational) \emph{C-finite sequences} (also known as \emph{C-recursive} sequences), and are characterized by the fact that they satisfy a linear recursion with constant rational coefficients. More explicitly, we define a sequence $A(n)$ of rational numbers to be C-finite if there exists a polynomial $P(x) \in \QQ[x]$ such that for every $n\geq0$ we have
\begin{equation}\label{def:cfinite0}
P(N) A(n) = 0,
\end{equation}
where $N$ denotes the shift operator $N^\ell(A(n)) \coloneqq A(n+\ell)$ for all $\ell \geq 0$. Equivalently, there exist integers $r>0$ and $n_0\geq0$, and complex numbers $c_0,\dots,c_{r-1}$ with $c_0 \neq 0$ such that
\begin{equation}\label{def:cfinite}
A(n+r) = c_{r-1} A(n+r-1) + \cdots + c_0 A(n) \quad \text{for all $n\geq n_0$}.
\end{equation}
We recall that associated to the recursion~\eqref{def:cfinite}, the \emph{characteristic roots} are usually defined as the roots of 
\[
\chi(\lambda) \coloneqq \lambda^r - c_{r-1} \lambda^{r-1} - \cdots - c_0.
\]
For our purpose, however, it is useful to define the characteristic roots of a C-finite sequence $A(n)$ as the roots of $P(x)$, where $P(x)$ is chosen with minimal degree such that \eqref{def:cfinite0} holds. Note that the only difference between considering roots of $\chi$ and $P$ is that $0$ can be a root of the latter. Equivalently, 0 is defined to be a characteristic root of $A(n)$ of multiplicity $m_0$ if the minimal $n_0$ in \eqref{def:cfinite} (chosen so that $r$ is minimal) equals $m_0$. With these definitions we obtain the following:

\begin{proposition}
  \label{thm:C:ct:1:pf}Let $A (n)$ be a C-finite sequence. $A (n)$ is a constant term if and only if it has a single characteristic root
  $\lambda$ and $\lambda \in \mathbb{Q}$.
\end{proposition}

This proposition immediately answers \cref{Q2} but also shows that, for example, the sequence $A(n) = 2^n+1$ is not a constant term sequence either (in both of these cases, there are two different characteristic roots). Evidently, however, it is the sum of two constant terms: we see that the class of constant term sequences is not a ring. Therefore, to fix this issue, it is natural to consider the class of sequences given as $\QQ$-linear combinations of constant terms: 

        \begin{question}	 \label{Q3}
\textit{Which P-finite sequences are finite $\QQ$-linear combinations of constant terms?}
    \end{question}

Again in the case of C-finite sequences, we can answer this question completely with the main result of the present work:
\begin{theorem}
  \label{thm:C:ct} Let $A (n)$ be a C-finite sequence. Then $A (n)$ is an
  $r$-term $\mathbb{Q}$-linear combination of constant terms if and only if it
  has at most $r$ distinct characteristic roots, all of which are rational.
\end{theorem}

Having completed the classification of C-finite sequences 
that can be written as (sums of) constant terms,
there are two most natural directions for further work. On the one hand, it is reasonable to go from diagonals in one variable to diagonals in two variables. By the combination of results due to Pólya~\cite{Polya22} and Furstenberg \cite{furstenberg-diag} this is
known to be 
exactly the class of algebraic generating functions. One is then lead to the following question which we leave for future work:

            \begin{question}	 \label{Q4}
    \textit{Which sequences $A(n)$ with algebraic generating function are constant terms?}
    \end{question}

Another reasonable direction is to try to classify those hypergeometric sequences which are constant terms. Recall that a P-finite sequence $A(n)$ is called \emph{hypergeometric} if it satisfies a recursion of order one, i.e. $\alpha(n) A(n+1) = \beta(n)A(n)$ for some polynomials $\alpha(n), \beta(n) \in \QQ[n]$. In this sense, this class of sequences is arguably the simplest (and best understood) among P-finite ones. Still, Christol's conjecture remains open even in this very special case. In fact, it is still an open question whether the generating function of the sequence 
\[
A(n) = \frac{\left(\tfrac{1}{9} \right)_n \left(\tfrac{4}{9} \right)_n
  \left(\tfrac{5}{9} \right)_n}{n!^2 \left(\tfrac{1}{3} \right)_n}
\]
can be represented as the diagonal of a rational function. Here and later, $(x)_n \coloneqq x(x+1)\cdots(x+n-1)$ denotes the rising factorial. We can use the same methods as in the C-finite case to prove that $A(n)$ is not a constant term sequence (see \cref{lem:christol_hypergeom}). By classifying when the family \eqref{eq:hyp:2f1:ct} of hypergeometric sequences is a constant term, we are further able to conclude that not all hypergeometric diagonals are constant terms. The following question, however, remains open in general:

    \begin{question}	 \label{Q5}
    \textit{Which hypergeometric sequences are constant terms?}
    \end{question}

\medskip

The organization of the paper is as follows: In \cref{sec:trace}, we review properties of C-finite sequences that will be important for our purposes. In particular, we state \cref{thm:minton} which is due to Minton~\cite{minton-cong} and which is a crucial ingredient of our approach. In \cref{sec:cong}, we derive certain congruences that are satisfied by any constant term sequence; these are already enough to answer \cref{Q2}. 
By combining these congruences with Minton's Theorem, we prove in \cref{sec:C:ct} our main \cref{thm:C:ct}, thus answering \cref{Q1} and \cref{Q3} in the case of C-finite sequences. 
In the short \cref{sec:minton:ct} we prove a statement which is pleasingly similar to Minton's theorem and which allows to classify the constant terms with $Q=1$ among all constant terms. Finally, in \cref{sec:hyp}, we turn our attention to hypergeometric sequences and discuss \cref{Q5}.

\medskip
Throughout the article, $p$ denotes a prime number, $\mathbb{F}_p$ the finite field with $p$ elements 
and $\mathbb{Z}_p$ the ring of $p$-adic integers.

\section{Trace sequences}\label{sec:trace}

Let $A (n)$ be a C-finite sequence. Denote by $\lambda_1, \lambda_2, \ldots, \lambda_d \in \overline{\mathbb{Q}}$ the characteristic roots, and let $m_j$ be the multiplicity of the root $\lambda_j$. Recall that $\lambda_0 = 0$ is defined to be a characteristic root of $A(n)$ of multiplicity $m_0$ if the minimal $n_0$ in \eqref{def:cfinite} equals $m_0$. $A (n)$ can be written
as a linear combination
\begin{equation}
  A (n) = A_0(n) + \sum_{j = 1}^d \sum_{r = 0}^{m_j - 1} c_{j, r} n^r \lambda_j^{n}
  \label{eq:C:lincomb}
\end{equation}
for certain coefficients $c_{j, r} \in \overline{\mathbb{Q}}$ (more precisely, $c_{j, r} \in \mathbb{Q} (\lambda_1, \ldots, \lambda_d)$) and $A_0(n)$ a sequence of finite support $\{0,1,\ldots,m_0-1\}$. We refer to \cite{epsw-rec} or \cite[Chapter~4]{kauers-paule-ct} 
for introductions to
C-finite sequences. Note that allowing $0$ as a characteristic root is equivalent to not restricting the numerator of the rational generating function of $A(n)$ to have degree less than the degree of its denominator. In the following, we will refer to
\begin{equation}
  A^{\operatorname{sep}} (n) = A_0(0) + \sum_{j = 1}^d c_{j, 0} \lambda_j^n
  \label{eq:C:sep:lincomb}
\end{equation}
as the \emph{separable part} of $A (n)$. We note that, if $A (n) \in
\mathbb{Q}$, then $A^{\operatorname{sep}} (n) \in \mathbb{Q}$.

A sequence $A (n)$ is said to be a \emph{trace sequence} if it is a
$\mathbb{Q}$-linear combination of traces $\operatorname{Tr} (\theta^n) = \theta_1^n +
\cdots + \theta_r^n$ of algebraic numbers $\theta$ with Galois conjugates
$\theta_1 = \theta, \theta_2, \ldots, \theta_r$ (with the understanding that $\operatorname{Tr}(0^n)$ is $1$ for $n=0$ and $0$ otherwise). Equivalently, a trace
sequence is a C-finite sequence for which the multiplicity of each characteristic root is $m_j = 1$ and for which $c_{i, 0} = c_{j, 0}$ in \eqref{eq:C:lincomb} whenever $\lambda_i$ and $\lambda_j$ have the same
minimal polynomial. We further note as in \cite{bhs-gauss} that the
condition to be a trace sequence is equivalent to the property that the
generating function $F (x)$ is $F (0)$ plus a $\mathbb{Q}$-linear combination
of functions of the form $x u' (x) / u (x)$, where $u \in \mathbb{Q} [x]$ is
irreducible and $u (0) = 1$.

\begin{example} \label{ex:fib}
  \label{eg:F:L:trace}For the Fibonacci numbers $F (n)$, the representation
  \eqref{eq:C:lincomb} takes the form
  \begin{equation} \label{eq:fib}
    F (n) = \frac{\varphi_+^n - \varphi_-^n}{\sqrt{5}}, \quad \varphi_{\pm} =
     \frac{1 \pm \sqrt{5}}{2} .
  \end{equation}
  Because the coefficients of $\varphi_+^n$ and $\varphi_-^n$ differ in sign,
  the Fibonacci numbers $F (n)$ are not a trace sequence. On the other hand,
  the Lucas numbers
  \begin{equation}
    L (n) = \varphi_+^n + \varphi_-^n = \operatorname{tr} [M^n], \quad M =
    \begin{bmatrix}
      0 & 1\\
      1 & 1
    \end{bmatrix}, \label{eq:lucas:L}
  \end{equation}
  which satisfy the same recurrence as the Fibonacci numbers, are a trace
  sequence. In particular, it follows from \cref{thm:minton} that the
  Lucas numbers $L (n)$ satisfy the Gauss congruences \eqref{eq:gauss}.
\end{example}

Minton~\cite{minton-cong} classified those C-finite sequences that satisfy
the Gauss congruences \eqref{eq:gauss} (see \cite{bhs-gauss} for another
proof of Minton's result). 

\begin{theorem}[Minton, 2014] \label{thm:minton}
Let $A (n)$ be C-finite. Then the following are
  equivalent:
  \begin{enumerate}
    \item For all large enough primes $p$ and for all $r\geq 1$, $A (n)$ satisfies the Gauss
    congruences
    \begin{equation}
      A (p^r n) \equiv A (p^{r - 1} n) \pmod{p^r} .
      \label{eq:gauss}
    \end{equation}
    \item For all large enough primes $p$, $A (n)$ satisfies the congruences
    \begin{equation}
      A (p n) \equiv A (n) \pmod{p} . \label{eq:gauss:r1}
    \end{equation}
    \item $A (n)$ is a trace sequence.
  \end{enumerate}
\end{theorem}

We conclude from Minton's \cref{thm:minton} the following result, which we employ in the proof of our main result (\cref{thm:C:ct}). To see the importance of \cref{lem:C:minton:sep}, we note that, as we will show later (in \cref{cor:ct:pr:large}), the sequences $A(n)$ which are linear combinations of constant terms satisfy the congruences $A (p^r n) \equiv A (p n) \pmod{p}$ for all $r\geq1$ and large enough primes~$p$. 

\begin{lemma}
  \label{lem:C:minton:sep}Let $A (n)$ be C-finite. If $A (n)$ satisfies the
  congruences
  \begin{equation}
    A (p^r n) \equiv A (p n) \pmod{p} \label{eq:A:pr:p}
  \end{equation}
  for all $r \geq 1$ and for all large enough primes $p$, then the separable
  part $A^{\operatorname{sep}} (n)$ is a trace sequence.
\end{lemma}

\begin{proof}
  It follows from comparing \eqref{eq:C:lincomb} with
  \eqref{eq:C:sep:lincomb} that for $n$ large enough
  \begin{equation*}
    A (n) = A^{\operatorname{sep}} (n) + n \tilde{A} (n), 
  \end{equation*}
  where $A^{\operatorname{sep}} (n)$ and $\tilde{A} (n)$ are rational and satisfy the
  minimal recurrence for $A (n)$. In particular, each of these sequences is in
  $\mathbb{Z}_p$ for large enough $p$, since denominators can only arise from the coefficients of the recurrence and the initial conditions. It follows that
  \begin{equation*}
    A^{\operatorname{sep}} (p n) \equiv A (p n) \pmod{p}
  \end{equation*}
  for all large enough $p$. Consequently, the congruences \eqref{eq:A:pr:p}
  are also satisfied by the C-finite sequence $A^{\operatorname{sep}} (n)$. That is,
  for all $r \geq 1$ and large enough $p$
  \begin{equation}
    A^{\operatorname{sep}} (p^r n) \equiv A^{\operatorname{sep}} (p n) \pmod{p}.
    \label{eq:A:sep:pr:p}
  \end{equation}
  On the other hand, let us consider the C-finite
  sequence $A^{\operatorname{sep}} (n)$ in $\mathbb{F}_p$. To avoid confusion, we denote this reduced sequence by $a_p^{\operatorname{sep}} (n)$. Since the
  characteristic polynomial of $A^{\operatorname{sep}} (n)$ over $\mathbb{Q}$ is
  separable, it is also separable for all large enough primes $p$ (this can be
  seen by looking at the discriminant which, if nonzero over $\mathbb{Q}$, can
  only vanish modulo finitely many primes). Consequently, we have a version of
  \eqref{eq:C:sep:lincomb} with coefficients and roots in
  $\overline{\mathbb{F}}_p$. Namely,
  \begin{equation*}
    a_p^{\operatorname{sep}} (n) = \sum_{j = 1}^d d_j \mu_j^n, \quad d_j, \mu_j \in
     \overline{\mathbb{F}}_p .
  \end{equation*}
  Denoting with $\varphi_p : \overline{\mathbb{F}}_p \rightarrow
  \overline{\mathbb{F}}_p$ the Frobenius automorphism defined by $\varphi_p (z) =
  z^p$, we therefore have
  \begin{equation*}
    a_p^{\operatorname{sep}} (p^s n) = \sum_{j = 1}^d d_j \mu_j^{p^s n} = \sum_{j =
     1}^d d_j (\varphi_p^s (\mu_j))^n
  \end{equation*}
  for each $s \in \mathbb{Z}_{> 0}$. Note that $\varphi_p$ acts as a
  permutation on the roots $\mu_j$. Writing $m$ for the order of this
  permutation, we have $\varphi_p^m (\mu_j) = \mu_j$ and thus
  \begin{equation*}
    a_p^{\operatorname{sep}} (p^m n) = a_p^{\operatorname{sep}} (n) .
  \end{equation*}
  Consequently, the corresponding sequence $A^{\operatorname{sep}} (n)$ satisfies
  \begin{equation*}
    A^{\operatorname{sep}} (p^m n) \equiv A^{\operatorname{sep}} (n) \pmod{p} .
  \end{equation*}
  Combined with the congruences \eqref{eq:A:sep:pr:p}, this implies that
  \begin{equation*}
    A^{\operatorname{sep}} (p n) \equiv A^{\operatorname{sep}} (n) \pmod{p}
  \end{equation*}
  for all large enough primes~$p$. \cref{thm:minton} therefore implies that
  $A^{\operatorname{sep}} (n)$ is a trace sequence.
\end{proof}

\section{Congruences for constant terms}\label{sec:cong}

In this section we will show that if $A(n)$ is a constant term sequence then it must satisfy certain congruences for large enough primes $p$. As a consequence, this allows us to conclude that the Fibonacci numbers are not a constant term sequence, thus answering \cref{Q2} from the introduction. 

For a Laurent polynomial $P \in \mathbb{Q} [\boldsymbol{x}^{\pm 1}]$, let $\deg
(P)$ denote the maximal degree with which any variable or its inverse appears
in $P$.

\begin{lemma}
  \label{lem:ct:pr:large}Let $A (n) = \operatorname{ct} [P (\boldsymbol{x})^n Q
  (\boldsymbol{x})]$ with $P, Q \in \mathbb{Z}_p [\boldsymbol{x}^{\pm 1}]$. Then
  \begin{equation*}
    A (p^r n + k) \equiv A (k) \operatorname{ct} [P (\boldsymbol{x})^{p^{r - 1} n}]
     \pmod{p^r}
  \end{equation*}
  for all integers $n, k \geq 0$ and $r \geq 1$, provided that $p >
  \deg (P^k Q)$.
\end{lemma}

\begin{proof}
  Recall that (see, for instance, \cite[Proposition~1.9]{ry-diag13}), for
  any Laurent polynomial $F \in \mathbb{Z}_p [\boldsymbol{x}^{\pm 1}]$,
  \begin{equation}
    F (\boldsymbol{x})^{p^r} \equiv F (\boldsymbol{x}^p)^{p^{r - 1}} \pmod{p^r} . \label{eq:power:pr}
  \end{equation}
  As in \cite{s-schemes}, it follows from \eqref{eq:power:pr} that
  \begin{eqnarray*}
    A (p^r n + k) & = & \operatorname{ct} [P (\boldsymbol{x})^{p^r n} P (\boldsymbol{x})^k
    Q (\boldsymbol{x})]\\
    & \equiv & \operatorname{ct} [P (\boldsymbol{x}^p)^{p^{r - 1} n} P (\boldsymbol{x})^k
    Q (\boldsymbol{x})] \pmod{p^r}\\
    & = & \operatorname{ct} [P (\boldsymbol{x})^{p^{r - 1} n} \Lambda_p [P
    (\boldsymbol{x})^k Q (\boldsymbol{x})]],
  \end{eqnarray*}
  where $\Lambda_p$ denotes the section operator
  \begin{equation*}
    \Lambda_p \left[ \sum_{\boldsymbol{k} \in \mathbb{Z}^d} a_{\boldsymbol{k}}
     \boldsymbol{x}^{\boldsymbol{k}} \right] = \sum_{\boldsymbol{k} \in
     \mathbb{Z}^d} a_{p\boldsymbol{k}} \boldsymbol{x}^{\boldsymbol{k}} .
  \end{equation*}
  If $p > \deg (P^k Q)$, then
  \begin{equation*}
    \Lambda_p [P (\boldsymbol{x})^k Q (\boldsymbol{x})] = \operatorname{ct} [P
     (\boldsymbol{x})^k Q (\boldsymbol{x})] = A (k)
  \end{equation*}
  and the claim follows.
\end{proof}

\begin{example} \label{ex:fibonacci}
  For the Fibonacci numbers $F(n)$, it is a well-known consequence of \eqref{eq:fib} that, modulo any prime $p$, we have the congruences
  \begin{align*}
    F(p) \equiv \begin{cases}
      1, &\text{if $p \equiv 1,4 \bmod{5}$}, \\
      -1, &\text{if $p \equiv 2,3 \bmod{5}$},
    \end{cases}
    \pmod{p}.
  \end{align*}
  Since this is incompatible with \cref{lem:ct:pr:large} (setting $r=n=1$ and $k=0$ implies that $A(p) \equiv A(0)\cdot c \pmod{p}$ for some $c \in \QQ$ that is independent of~$p$), we see that $F(n)$ is not a constant term sequence.

  On the other hand, by \cref{thm:minton}, the Lucas numbers $L(n)$ (from~\eqref{eq:lucas:L}) satisfy the congruences $L(p^r n) \equiv L(p n) \pmod{p}$ for $r\geq1$ and $p$ large enough.
  As such, \cref{lem:ct:pr:large} is not sufficient to conclude that $L(n)$ is not a constant term sequence.
  However, we will be able to conclude in \cref{ex:lucas} the stronger result that both the Fibonacci numbers and the Lucas numbers cannot be expressed as a $\QQ$-linear combination of constant terms.
\end{example}

\begin{corollary}
  Let $A (n) = \operatorname{ct} [P (\boldsymbol{x})^n Q (\boldsymbol{x})]$ with $P, Q \in
  \mathbb{Z}_p [\boldsymbol{x}^{\pm 1}]$. Then
  \begin{equation*}
    A (p^s n + k) \equiv A (p^r n + k) \pmod{p^r}
  \end{equation*}
  for all integers $n, k \geq 0$ and $s \geq r \geq 1$,
  provided that $p > \deg (P^k Q)$.
\end{corollary}

\begin{proof}
  It follows from \cref{lem:ct:pr:large} and \eqref{eq:power:pr} that
  \begin{eqnarray*}
    A (p^s n + k) & \equiv & A (k) \operatorname{ct} [P (\boldsymbol{x})^{p^{s - 1} n}]
    \pmod{p^s}\\
    & \equiv & A (k) \operatorname{ct} [P (\boldsymbol{x}^{p^{s - r}})^{p^{r - 1} n}]
    \pmod{p^r}\\
    & = & A (k) \operatorname{ct} [P (\boldsymbol{x})^{p^{r - 1} n}],
  \end{eqnarray*}
  as claimed.
\end{proof}

The simple but useful special case $r = 1$ and $k = 0$ of the corollary above takes the following
form. Here, $p$ is large enough if $p > \deg (Q)$ and $P, Q \in \mathbb{Z}_p
[\boldsymbol{x}^{\pm 1}]$.

\begin{corollary}
  \label{cor:ct:pr:large}Let $A (n) = \operatorname{ct} [P (\boldsymbol{x})^n Q
  (\boldsymbol{x})]$ with $P, Q \in \mathbb{Q} [\boldsymbol{x}^{\pm 1}]$. If $p$
  is a large enough prime, then, for all integers $n \geq 0$ and $r \geq 1$,
  \begin{equation*}
    A (p^r n) \equiv A (p n) \pmod{p} .
  \end{equation*}
\end{corollary}

\section{\texorpdfstring{$C$}{C}-finite sequences that are constant terms}\label{sec:C:ct}

In this section, we prove our main result, \cref{thm:C:ct} stated in the introduction. We thus classify those C-finite sequences that are constant terms or linear combinations of such. We start by proving the following weaker version, since it illustrates well our approach and the usefulness of the congruences proved in \cref{sec:cong}. We then extend the argument to prove \cref{thm:C:ct} in full generality.

\begin{proposition} \label{prop:charroots:rational}
  Let $A (n)$ be a C-finite sequence. $A (n)$ is a $\mathbb{Q}$-linear
  combination of constant terms if and only if all characteristic roots are
  rational.
\end{proposition}

\begin{proof}
  For one direction, note that
  \begin{equation}
    \operatorname{ct} [(x + \lambda)^{n} (\lambda/x)^{r}] = \binom{n}{r} \lambda^{n} =
    \frac{n (n - 1) \cdots (n - r + 1)}{r!} \lambda^{n}. \label{eq:ct:exp}
  \end{equation}
  Varying $r$, the right-hand side forms a basis for the span of the sequences $(n^r \lambda^{n})_{n \geq 0}$. A sequence $A_0(n)$ of finite support can be represented as \[
    A_0(n) = \ct[x^n (A(0)+A(1)x^{-1}+\cdots+A(N)x^{-N})],
  \]
  where $N$ is the largest integer for which $A_0(N)$ is non-zero. It therefore follows with \eqref{eq:C:lincomb} that, if all characteristic roots $\lambda$ are rational, then $A (n)$ can
  be represented as a linear combination of constant terms.
  
  On the other hand, suppose that $A (n)$ is a linear combination of constant
  terms. Note that this implies that any shift $A (n + k)$, where $k \in
  \mathbb{Z}_{\geq 0}$, is a linear combination of constant terms as
  well. These shifts generate the space $V_A$ of rational solutions of the
  minimal constant-coefficient recursion satisfied by $A (n)$. Thus, any
  sequence in $V_A$ is a linear combination of constant terms. Assume, for
  contradiction, that there is a characteristic root $\lambda$ that is not
  rational. Then among the sequences in $V_A$ there is always a sequence $B (n)$ of
  the form \eqref{eq:C:sep:lincomb} (that is, $B (n)$ equals its separable
  part) which is not a trace sequence.
  
  For instance, if $\lambda_1, \ldots, \lambda_d$ are the roots of the minimal
  polynomial of~$\lambda$, then the space $V_{\lambda}$ of rational sequences
  of the form $b (n) = c_1 \lambda_1^n + \cdots + c_d \lambda_d^n$, with $c_1,
  \ldots, c_d \in \overline{\mathbb{Q}}$, is a $d$-dimensional subspace of $V_A$.
  Clearly, each sequence in $V_{\lambda}$ is of the form
  \eqref{eq:C:sep:lincomb}. Note that $\lambda_1^n + \cdots + \lambda_d^n$ and
  its multiples are the only trace sequences in $V_{\lambda}$. Since $d \geq 2$,
  we can therefore choose a sequence $B (n)$ in $V_{\lambda}$ that is not a
  trace sequence.
  
  It follows from \cref{cor:ct:pr:large} that $B (n)$ satisfies the
  congruences
  \begin{equation*}
    B (p^r n) \equiv B (p n) \pmod{p}
  \end{equation*}
  for all $r \geq 1$ and all large enough primes $p$.
  \cref{lem:C:minton:sep} therefore implies that $B^{\operatorname{sep}} (n) = B
  (n)$ is a trace sequence. This is a contradiction, and we conclude that all
  characteristic roots must be rational.
\end{proof}

\begin{example}\label{ex:lucas}
  Recall from \eqref{eq:fib} that the Fibonacci numbers $F(n)$ are C-finite with characteristic roots $(1\pm\sqrt{5})/2$.
  Since these are not rational, it follows from \cref{prop:charroots:rational} that $F(n)$ cannot be expressed as a linear combination of constant terms.

  The same argument applied to \eqref{eq:lucas:L} shows that the Lucas numbers $L(n)$ cannot be expressed as a linear combination of constant terms either.
  Alternatively, this can also be concluded from the relationship
    \[
    2L(n+1) - L(n) = 5F(n)
    \]
  combined with the fact that Fibonacci numbers are not a sum of constant terms.
\end{example}

We next prove the case $r = 1$ of \cref{thm:C:ct}, that is \cref{thm:C:ct:1:pf}, stating that a C-finite sequence $A (n)$ is a single constant term if and only if it has a single characteristic root $\lambda$ and $\lambda \in \mathbb{Q}$.

\begin{proof}[Proof of \cref{thm:C:ct:1:pf}]
  It follows from \eqref{eq:C:lincomb} and \eqref{eq:ct:exp} that if $A (n)$
  is a C-finite sequence with the single characteristic root $\lambda \in
  \mathbb{Q}$ (possibly repeated or possibly 0), then $A (n)$ is a constant term, namely $A
  (n) = \operatorname{ct} [(x + \lambda)^n Q (x^{- 1})]$ for a suitable polynomial $Q
  (x)$.
  
  On the other hand, suppose that $A (n) = \operatorname{ct} [P (\boldsymbol{x})^n Q
  (\boldsymbol{x})]$ is a single constant term. Since $A (n)$ is a C-finite
  sequence, it has a representation of the form \eqref{eq:C:lincomb} or,
  equivalently,
  \begin{equation}
    A (n) = A_0 (n) + \sum_{j = 1}^d \lambda_j^n p_j (n)
    \label{eq:C:lincomb:poly}
  \end{equation}
  for pairwise distinct $\lambda_j \in \overline{\mathbb{Q}}^{\times}$ and nonzero
  $p_j (n) \in \overline{\mathbb{Q}} [n]$. As before, $A_0 (n)$ is a sequence with
  finite support, corresponding to the characteristic root $0$. It follows from \cref{prop:charroots:rational} that all characteristic roots
  $\lambda_j$ are rational, and this further implies that $p_j (n) \in
  \mathbb{Q} [n]$.
  
  Let $c_0 = \operatorname{ct} [P (\boldsymbol{x})] \in \mathbb{Q}$. From
  \cref{lem:ct:pr:large} (with $r = 1$ and $n = 1$) it follows that
  \begin{equation*}
    A (p + n) \equiv A (n) \cdot c_0 \pmod{p}
  \end{equation*}
  for all $n \geq 0$ and all large enough primes $p$ (namely, $p > \deg
  (P^n Q)$ and large enough so that $c_0 \in \mathbb{Z}_p$). Combining this
  congruence with \eqref{eq:C:lincomb:poly} and applying Fermat's little
  theorem to reduce $\lambda_j^{p + n}$ and $p_j (p + n)$ modulo $p$ to
  $\lambda_j^{n + 1}$ and $p_j (n)$ respectively, we find that
  \begin{equation}
    \sum_{j = 1}^d \lambda_j^{n + 1} p_j (n) \equiv c_0 \left[ A_0 (n) +
    \sum_{j = 1}^d \lambda_j^n p_j (n) \right] \pmod{p}
    \label{eq:Apn:modp}
  \end{equation}
  for all large enough $p$ (in particular, so that $p$ is larger than any
  denominator occuring in the $p_j (n)$ and so that $A_0 (p + n) = 0$). Note
  that both sides of~\eqref{eq:Apn:modp} are independent of $p$. Since they
  agree modulo any large enough $p$, it follows that they must be equal (for
  each fixed value of $n$). Accordingly, we have the identity
  \begin{equation}
    \sum_{j = 1}^d \lambda_j^{n + 1} p_j (n) = c_0 \left[ A_0 (n) + \sum_{j =
    1}^d \lambda_j^n p_j (n) \right] \label{eq:Apn} \quad \text{for all $n \geq 0$.}
  \end{equation}
  Note that both sides of \eqref{eq:Apn} are
  C-finite sequences so that, because the representation
  \eqref{eq:C:lincomb:poly} is unique, we must have, in particular, $c_0 A_0
  (n) = 0$. If $c_0 = 0$ then it follows by comparison with the left-hand side
  of \eqref{eq:Apn} that $d = 0$ so that $A (n) = A_0 (n)$ with the single
  characteristic root $\lambda = 0$. In the other case, that is if $c_0 \neq 0$, we have $A_0 (n) = 0$, so $0$ is not a characteristic root. Further comparing both sides of
  \eqref{eq:Apn}, we find that $\lambda_j = c_0$ for all $j$. Since the
  $\lambda_j$ are distinct, we conclude that $d = 1$ so that $A (n) =
  \lambda_1^n p_1 (n)$ with the single characteristic root $\lambda_1 \in
  \mathbb{Q}^{\times}$.
\end{proof}

We now extend \cref{thm:C:ct:1:pf} to the case of $r$-term
$\mathbb{Q}$-linear combinations of constant terms, thus proving our main result, \cref{thm:C:ct}. We recall that its statement is that a C-finite $A (n)$ sequence is an $r$-term $\mathbb{Q}$-linear combination of constant terms if and only if it has at most $r$ characteristic roots, all of which are rational.

\begin{proof}[Proof of \cref{thm:C:ct}]
  The case $r = 1$ is proved by \cref{thm:C:ct:1:pf}. With the same argument as in \eqref{eq:ct:exp} it follows that any C-finite sequence with $r$ characteristic roots, all of which are rational, can be
  represented as a linear combination of $r$ constant terms.
  
  Therefore, suppose that $r > 1$ and that
  \begin{equation*}
    A (n) = \operatorname{ct} [P_1 (\boldsymbol{x})^n Q_1 (\boldsymbol{x})] + \cdots +
     \operatorname{ct} [P_r (\boldsymbol{x})^n Q_r (\boldsymbol{x})]
  \end{equation*}
  is an $r$-term $\mathbb{Q}$-linear combination of constant terms with $P_j,
  Q_j \in \mathbb{Q} [\boldsymbol{x}^{\pm 1}]$. We need to show that $A (n)$ has
  at most $r$ characteristic roots, all of which are rational. As in the proof
  of \cref{thm:C:ct:1:pf}, we find that all characteristic roots of $A
  (n)$ are rational and that $A (n)$ can be represented in the form
  \eqref{eq:C:lincomb:poly} with $p_j (n) \in \mathbb{Q} [n]$.
  
  Let $c_j = \operatorname{ct} [P_j (\boldsymbol{x})] \in \mathbb{Q}$. It follows from
  \cref{lem:ct:pr:large} that
  \begin{equation*}
    A (p + n) \equiv c_1 \operatorname{ct} [P_1 (\boldsymbol{x})^n Q_1 (\boldsymbol{x})]
     + \cdots + c_r \operatorname{ct} [P_r (\boldsymbol{x})^n Q_r (\boldsymbol{x})] \pmod{p}
  \end{equation*}
  for all $n \geq 0$ and all large enough primes $p$. On the other hand, for large $p$,
  by Fermat's little theorem,
  \begin{equation*}
    A (p + n) \equiv \sum_{j = 1}^d \lambda_j^{n + 1} p_j (n) \pmod{p} .
  \end{equation*}
  Note that the right-hand sides of the last two congruences are independent
  of~$p$. Since the congruences hold modulo all large enough primes, we
  conclude that
  \begin{equation*}
    \sum_{j = 1}^d \lambda_j^{n + 1} p_j (n) = c_1 \operatorname{ct} [P_1
     (\boldsymbol{x})^n Q_1 (\boldsymbol{x})] + \cdots + c_r \operatorname{ct} [P_r
     (\boldsymbol{x})^n Q_r (\boldsymbol{x})] .
  \end{equation*}
  Note that the sequence
  \begin{equation*}
    B (n) \coloneqq \sum_{j = 1}^d \lambda_j^{n + 1} p_j (n) - c_1 A (n) =
     \sum_{j = 1}^d (\lambda_j - c_1) \lambda_j^n p_j (n) - c_1 A_0 (n)
  \end{equation*}
  is C-finite and is an $(r - 1)$-term $\mathbb{Q}$-linear combination of
  constant terms. By induction, we may conclude that $B (n)$ has at most $r -
  1$ characteristic roots, all of which are rational. By comparison with
  \eqref{eq:C:lincomb:poly}, we see that $A (n)$ has at most one more
  characteristic root than $B (n)$. Thus $A (n)$ has at most $r$
  characteristic roots, which is what we had to show.
\end{proof}

\cref{thm:C:ct} classifies those rational recursive sequences with constant coefficients which can be represented as a linear combination of $r$ constant terms. In particular, a rational C-finite sequence $A (n)$ is a linear combination of constant terms if and only if all of its characteristic roots are rational. It is natural to wonder whether we can restrict to integer sequences and conclude that all characteristic roots must be integral. This can be achieved by using the following proposition\footnote{The proof of Prop.~\ref{conj:new} was communicated to us by \href{https://sites.google.com/view/carlo-sanna-math}{Carlo Sanna} (Politecnico di Torino).},
that we could not locate in the vast literature on C-finite sequences.
\begin{proposition}\label{conj:new}
  Let $A (n)$ be a C-finite sequence with characteristic roots $\lambda_1,
  \ldots, \lambda_d \in \mathbb{Q}$. If $A (n)$ is an integer sequence, then
  $\lambda_1, \ldots, \lambda_d \in \mathbb{Z}$.
\end{proposition}
\begin{proof}
By assumption, $A(n)$ is equal to $\sum_{i=1}^d p_i(n) \lambda_i^n$, where the $\lambda_i$'s are mutually distinct rational numbers and the $p_i(x)$'s are polynomials in $\QQ[x]$. We will prove that if for some $N\in\ZZ\setminus\{ 0 \}$ we have that $A(n)\in\frac{1}{N}\ZZ$ for all~$n$, then all the $\lambda_i$'s are integers. 

Let us start with the observation that this is true if $d=1$; indeed, if $\lambda_1 = a/b$ with coprime integers $a,b$, and $p_1(n)=q(n)/c$ with $q(x)\in\ZZ[x]$ and $c\in \ZZ$, then the assumption ``$A(n) = p_1(n)\lambda_1^n\in \frac{1}{N}\ZZ$ for all $n$'' implies
that $b^n$ divides $N q(n)$ for all $n$, hence $b=1$.

Let us now treat the general case. Denote by $V$ the $d\times d$ Vandermonde matrix attached with the $\lambda_i$'s, that is $V=(\lambda_i^{j-1})_{1\leq i,j \leq d}$. Since the $\lambda_i$'s are mutually distinct, $V$ is in $\text{GL}_d(\QQ)$. Therefore, the equality
\[
\left[ p_1(n)\lambda_1^n, \ldots, p_d(n)\lambda_d^n \right] \cdot V = 
\left[ A(n), \ldots, A(n+d-1) \right]
\]
implies that each term $p_i(n) \lambda_i^n$ is equal to $1/\det(V)$ times a linear combination of the integers $A(n), \ldots,  A(n+d-1)$ with fixed rational coefficients. In other terms,
each $p_i(n) \lambda_i^n$ is in $\frac{1}{N_i} \ZZ$ for some $N_i \in \ZZ\setminus \{ 0 \}$ independent of $n$.
By the case $d=1$, this implies that all the $\lambda_i$'s are integers.
\end{proof}

\begin{remark} \label{rem:fatou}
\cref{conj:new} also follows\footnote{We are indebted to the anonymous referee for pointing out this connection.} from ``Fatou's lemma'' (not to be confused with Fatou's famous result in Lebesgue integration theory)
 which states that if $f(x) \in \ZZ\ps{x}$ is a rational function, then one can write it as $f(x) = P(x)/Q(x)$ with $P,Q \in \ZZ[x]$ and $Q(0)=1$. This lemma is stated in Fatou's 1904 short communication~\cite[p.~313]{Fatou1904} and proved in his PhD thesis~\cite[p.~369]{Fatou1906}

To see how \cref{conj:new} follows from Fatou's lemma, note that since the generating function of the sequence $A(n)$ is rational with integer coefficients, we may represent it as in the lemma. The characteristic roots $\lambda_1,\dots,\lambda_d$ are the zeros of the reversal of $Q(x)$ which is monic because $Q(0)=1$. By Gauss' lemma it follows that the rational numbers $\lambda_1,\dots,\lambda_d$ are integers.
\end{remark}

We conclude this section with the following immediate consequence of \cref{thm:C:ct:1:pf}, \cref{thm:C:ct} and \cref{conj:new}:
\begin{corollary}
Let $A (n) \in \ZZ$ be a C-finite sequence. $A (n)$ is a constant term if and only if it has a single characteristic root   $\lambda$ and $\lambda \in \mathbb{Z}$. More generally, $A (n)$ is an $r$-term $\mathbb{Q}$-linear combination of constant terms if and only if it has at most $r$ distinct characteristic roots, all of which are integral.
\end{corollary}

\section{An analog of Minton's theorem}\label{sec:minton:ct}

In this section, we record the following result which, though having a much simpler proof, is pleasingly similar to \cref{thm:minton} due to Minton \cite{minton-cong}. Moreover, this result gives a classification of constant
term sequences of the form $A (n) = \operatorname{ct} [P (\boldsymbol{x})^n]$ among all
constant term sequences $\operatorname{ct} [P (\boldsymbol{x})^n Q (\boldsymbol{x})]$.

\begin{proposition}
  \label{lem:minton:ct}Suppose $A (n) = \operatorname{ct} [P (\boldsymbol{x})^n Q
  (\boldsymbol{x})]$ with $P, Q \in \mathbb{Q} [\boldsymbol{x}^{\pm 1}]$. Then the
  following are equivalent:
  \begin{enumerate}
    \item For all large enough primes $p$ and for all $r\geq 1$, $A (n)$ satisfies the Gauss
    congruences~\eqref{eq:gauss}.
    
    \item For all large enough primes $p$, $A (n)$ satisfies the
    congruences~\eqref{eq:gauss:r1}.
    
    \item $A (n) = A (0) \operatorname{ct} [P (\boldsymbol{x})^n]$.
  \end{enumerate}
\end{proposition}

\begin{proof}
  We conclude from \cref{lem:ct:pr:large} with $r = 1$ and $k = 0$ that
  \begin{equation*}
    A (p n) \equiv A (0) \operatorname{ct} [P (\boldsymbol{x})^n] \pmod{p}
  \end{equation*}
  for large enough $p$ (namely, if $p > \deg (Q)$). If $A (n)$ satisfies the
  congruences~\eqref{eq:gauss:r1}, we find that, for large enough $p$,
  \begin{equation*}
    A (n) \equiv A (0) \operatorname{ct} [P (\boldsymbol{x})^n] \pmod{p}.
  \end{equation*}
  In that case, since this congruence holds modulo
  infinitely many $p$, we conclude the equality $A (n) = A (0) \operatorname{ct} [P
  (\boldsymbol{x})^n]$. Thus the third condition follows from the second.
  
  To complete the proof, we need to show that the third condition implies the
  first. This follows from \cref{lem:ct:pr:large} with $k = 0$ and $Q =
  1$.
\end{proof}

\begin{remark}
  Note that \cref{lem:minton:ct} does not imply that if $A (n) = \operatorname{ct}
  [P (\boldsymbol{x})^n Q (\boldsymbol{x})]$ satisfies the Gauss
  congruences~\eqref{eq:gauss} for large enough primes, then $Q$ must be constant. For instance, for any $P (x) \in \mathbb{Z} [x^{\pm 1}]$, the constant terms $\operatorname{ct} [P (x^2)^n (1 + x)] = \operatorname{ct} [P (x^2)^n] = \operatorname{ct} [P (x)^n]$ satisfy the Gauss congruences for all primes $p$, even though the
  first constant term has a non-constant $Q$. \cref{lem:minton:ct} rather shows that if (a) or (b) are fulfilled, then $Q$ can be replaced by $\operatorname{ct} [Q]$.
\end{remark}

\section{Hypergeometric constant terms}\label{sec:hyp}

Exiting the class of C-finite sequences, we find it natural to ask (\cref{Q5} in the introduction): Which hypergeometric sequences\footnote{Recall that a sequence $A (n)$ is hypergeometric if it satisfies a first-order recurrence $\alpha (n) A (n + 1) = \beta (n) A (n)$ for some
polynomials $\alpha (n), \beta (n) \in \mathbb{Q} [n]$. For our purposes we will assume $\alpha (n) \neq 0$ for all $n \geq 0$.} $A (n)$ are constant term sequences?

The reason for the specialization to hypergeometric sequences is threefold. First, it can be argued that it is the easiest P-finite case. Second, similar to constant terms, hypergeometric sequences are not stable under addition. Finally, as we will see below in \cref{lem:ct:p:large:0:c}, the congruences proven in \cref{sec:cong} behave nicely with the hypergeometric assumption.\\

It follows from \cref{lem:ct:pr:large} (specialized to $n = 1$ and $r =
1$) that if $A (n) = \operatorname{ct} [P (\boldsymbol{x})^n Q (\boldsymbol{x})]$ with $P,
Q \in \mathbb{Q} [\boldsymbol{x}^{\pm 1}]$, then
\begin{equation*}
  A (p + k) \equiv A (k) \operatorname{ct} [P (\boldsymbol{x})] \pmod{p}
\end{equation*}
for all integers $k \geq 0$, provided that $p > \deg (P^k Q)$ and $P, Q
\in \mathbb{Z}_p [\boldsymbol{x}^{\pm 1}]$. In other words, if $A (n)$ is a
constant term, then there exists a constant $c \in \mathbb{Q}$ such that, for
each $k \in \mathbb{Z}_{\geq 0}$, the congruences
\begin{equation}
  A (p + k) \equiv A (k) \cdot c \pmod{p} \label{eq:ct:p:large:k:c}
\end{equation}
hold for all large enough primes~$p$. We shall now show that, for hypergeometric
sequences, the congruences \eqref{eq:ct:p:large:k:c} follow from the base case
$k = 0$.

\begin{lemma}\label{lem:ct:p:large:0:c}
Let $A (n)$ be a hypergeometric sequence.
  Suppose that there exists a constant $c \in \mathbb{Q}$ such that
  \begin{equation}
    A (p) \equiv c \pmod{p} \label{eq:ct:p:large:0:c}
  \end{equation}
  for all large enough primes~$p$. Then, for each $k \in \mathbb{Z}_{\geq 0}$,
  the congruence \eqref{eq:ct:p:large:k:c} holds for sufficiently large primes~$p$.
\end{lemma}

\begin{proof}
  Since $A (n)$ is hypergeometric, we have $A (n + 1) = \rho (n) A (n)$ for a
  rational function $\rho (n) = \beta (n) / \alpha (n)$ with $\alpha (n), \beta
  (n) \in \mathbb{Z} [n]$. Fix $k \in \mathbb{Z}_{\geq 0}$ and suppose
  that the congruence \eqref{eq:ct:p:large:k:c} holds for all large enough primes~$p$. By applying the hypergeometric recurrence twice, we obtain
  \begin{equation*}
    A (p + k + 1) = \rho (p + k) A (p + k) \equiv \rho (k) c A (k) = c A (k + 1)
     \pmod{p},
  \end{equation*}
  which is \eqref{eq:ct:p:large:k:c} with $k+1$ in place of $k$.
  Here we used that $\rho (p + k) \equiv \rho (k) \pmod{p}$,
  which holds true provided that $\alpha (k) \nequiv 0 \pmod{p}$. The latter is true for all sufficiently large primes~$p$ since, by
  assumption, $\alpha (k) \neq 0$. The claim therefore follows by induction on $k$.
\end{proof}

\begin{remark}
  Note that \cref{lem:ct:p:large:0:c} does not hold for
  non-hypergeometric sequences in general. For instance, it does not hold for
  the Lucas numbers $L (n)$ as defined in \eqref{eq:lucas:L}. These form a
  trace sequence so that, by Minton's \cref{thm:minton}, the Gauss
  congruences \eqref{eq:gauss} are satisfied. In the case $n = 1$, these imply
  the congruences \eqref{eq:ct:p:large:0:c}. However, the Lucas numbers do not
  satisfy the congruences \eqref{eq:ct:p:large:k:c} for $k > 0$.
\end{remark}

\cref{lem:ct:pr:large} gives a necessary condition for $A (n)$ to be a constant term sequence. It is natural to wonder whether, or to what extent, this condition is sufficient: Is any integer hypergeometric sequence $A(n)$ that satisfies the congruences \eqref{eq:ct:p:large:0:c} a constant term? Natural sources of potential counterexamples to this question are families of integer sequences that are quotients of binomial coefficients but cannot be written as products of those, for example $A(n) = \binom{8n}{4n} \binom{4n}{n}\binom{2n}{n}^{-1}$
(see~\cite[Thm. 1.2]{Bober09}).

We recall that the corresponding question for diagonals \eqref{eq:diag} (namely, to classify
which hypergeometric sequences $A (n)$ are coefficients of diagonals) also remains open. The following conjecture due to Christol~\cite[Conjecture 4, p.~55]{christol-glob} attempts such a classification. 
In its statement, we call a sequence $(A(n))_{n\geq0}$ \emph{almost integral} if there exists a positive integer $K$ such that $K^{n+1} A(n) \in \mathbb{Z}$ for all integers $n\geq0$.  An almost integral sequence with (at most) geometric growth is called \emph{globally bounded}.

\begin{conjecture}[\cite{christol-glob}]
  \label{conj:christol} Let $(A(n))_{n\geq0}$ be a sequence of rational numbers. The generating function $  \sum_{n\geq0} A(n)t^n$ is the diagonal of a rational function if and only if $(A(n))_{n\geq0}$ is P-finite and globally bounded.
\end{conjecture}

Any hypergeometric sequence is P-finite since it satisfies, by definition, a recurrence with polynomial coefficients of order one. Moreover, thanks to a result of Christol~\cite{christol-hyp-glob,christol-glob} it is easy to check when a hypergeometric sequence is integral (in the case when $\alpha(n)$ and $\beta(n)$ in the definition split in $\QQ[n]$). This makes hypergeometric sequences a natural source of potential counterexamples to \cref{conj:christol}. We refer to \cite{bbchm-diag}, \cite{akm-christol} and \cite{by-hyp-diag} for recent progress in this area. Here, we only mention that even for 
\begin{equation} \label{eq:christolsexample}
A(n) = \frac{\left(\tfrac{1}{9} \right)_n \left(\tfrac{4}{9} \right)_n \left(\tfrac{5}{9} \right)_n}{n!^2 \left(\tfrac{1}{3} \right)_n}
\end{equation}
the conjecture is open. In other words, it is an open question whether the sequence \eqref{eq:christolsexample} is the diagonal of a rational function. On the other hand, we will show in this section that \eqref{eq:christolsexample} is not a constant term. Before doing so, we first prove the following result answering \cref{Q5} for a special family of hypergeometric sequences.
\begin{lemma}
  \label{lem:hyp:2f1:ct}Let $m \geq 2$ be an integer and consider the
  sequence
  \begin{equation}
    A_m (n) = \frac{\left(\tfrac{1}{m} \right)_n \left(1 - \tfrac{1}{m}
    \right)_n}{n!^2} . \label{eq:hyp:2f1:ct}
  \end{equation}
  \begin{enumerate}
    \item $A_m (n)$ is a diagonal for all $m \geq 2$.
    
    \item $A_m (n)$ is a constant term if and only if $m \in \{ 2, 3, 4, 6
    \}$.
  \end{enumerate}
\end{lemma}

Note that the classification in \cref{lem:hyp:2f1:ct} suggests that constant term
sequences are special among diagonals and often have significant additional
arithmetic properties. Indeed, the cases $m \in \{ 2, 3, 4, 6 \}$ (see \href{http://oeis.org/A002894}{A002894}, \href{http://oeis.org/A006480}{A006480}, \href{http://oeis.org/A000897}{A000897} and \href{http://oeis.org/A113424}{A113424} in the on-line encyclopedia of integer sequences~\cite{OEIS}) correspond
precisely to those special hypergeometric functions underlying Ramanujan's
theory of elliptic functions ($m = 2$ being the classical case and $m = 3,
4, 6$ corresponding to the alternative bases). We refer to \cite{bbg-rama-ell}
for more information.

\begin{example}
  The hypergeometric sequence
  \begin{equation}
    B (n) = 5^{3 n} \frac{\left(\tfrac{1}{5} \right)_n \left(\tfrac{4}{5}
    \right)_n}{n!^2} = 1, 20, 1350, 115500, 10972500, \ldots
    \label{eq:hyp:2f1:5}
  \end{equation}
  is an integer sequence and grows at most exponentially. As suggested by
  Christol's \cref{conj:christol} and stated in \cref{lem:hyp:2f1:ct}, the sequence $B (n)$ is a diagonal. However, $B (n)$ is not a constant term. The
  proof of \cref{lem:hyp:2f1:ct} in this case proceeds by showing that we have the
  congruences
  \begin{equation*}
    B (p) \equiv \begin{cases}
       20, & \text{if $p \equiv \pm 1 \bmod{5}$},\\
       30, & \text{otherwise},
     \end{cases} \pmod{p} ,
  \end{equation*}
  which contradict \cref{lem:ct:p:large:0:c}. 
\end{example}

\begin{proof}[Proof of \cref{lem:hyp:2f1:ct}]
  Part (a) follows from the fact
  that the generating function of $A_m (n)$ is the Hadamard (term-wise) product of $(1 -
  x)^{- 1 / m}$ and $(1 - x)^{1 / m - 1}$. The latter are algebraic functions
  and hence diagonals by a result of Furstenberg \cite{furstenberg-diag}.
  Since diagonals are closed under Hadamard products \cite{christol-diag88},
  it follows that $A_m (n)$ is a diagonal.
  
  That $A_m (n)$ is a constant term if $m \in \{ 2, 3, 4, 6 \}$ follows from
  the following alternative representations as products of binomial
  coefficients:
  \begin{eqnarray*}
    2^{4 n} A_2 (n) & = & \frac{(2 n) !^2}{n!^4} = \binom{2 n}{n}^2,\\
    3^{3 n} A_3 (n) & = & \frac{(3 n) !}{n!^3} = \binom{3 n}{2 n} \binom{2
    n}{n},\\
    4^{3 n} A_4 (n) & = & \frac{(4 n) !}{(2 n) !n!^2} = \binom{4 n}{2 n}
    \binom{2 n}{n},\\
    2^{4 n} 3^{3 n} A_6 (n) & = & \frac{(6 n) !}{(3 n) ! (2 n) !n!} = \binom{6
    n}{3 n} \binom{3 n}{n}.
  \end{eqnarray*}
  In the remainder, we will show that $A_m (n)$ is not a constant term if $m
  \not\in \{ 2, 3, 4, 6 \}$. If $m$ is coprime to $p$ (as it is for large enough primes~$p$), then the right-hand side of
  \begin{equation*}
    m^p \left(\tfrac{1}{m} \right)_p = 1 \cdot (m + 1) (2 m + 1) \cdots ((p
     - 1) m + 1)
  \end{equation*}
  is a product of all the residues modulo $p$. In particular, exactly one
  factor is of the form $a p$ where $a \in \{ 1, 2, \ldots, m - 1 \}$ is
  characterized by $a p \equiv 1 \pmod{m}$. By Wilson's
  theorem, we therefore have
  \begin{equation*}
    m^p \left(\tfrac{1}{m} \right)_p \equiv - a p \pmod{p^2}
  \end{equation*}
  or, equivalently,
  \begin{equation*}
    \frac{m^p \left(\tfrac{1}{m} \right)_p}{p!} \equiv a \pmod{p} .
  \end{equation*}
  Similarly,
  \begin{equation*}
    m^p \left(1 - \tfrac{1}{m} \right)_p = (m - 1) (2 m - 1) \cdots (p m -
     1)
  \end{equation*}
  and, again, the right-hand side features a product of all residues modulo
  $p$. Exactly one factor is of the form $b p$ where $b \in \{ 1, 2, \ldots, m
  - 1 \}$ is characterized by $b p \equiv - 1 \pmod{m}$. It
  follows that $b = m - a$. Combined, we conclude that
  \begin{equation}
    m^{2 p} A_m (p) = \frac{m^{2 p} \left(\tfrac{1}{m} \right)_p \left(1 -
    \tfrac{1}{m} \right)_p}{p!^2} \equiv a (m - a) \pmod{p} .
    \label{eq:hyp:2f1:modp}
  \end{equation}
  Since $a \in \{ 1, 2, \ldots, m - 1 \}$ is characterized by $a \equiv
  1 / p \pmod{m}$ it, in particular, depends only on the
  residue class of $p$ modulo $m$. As $p$ ranges through all primes, it
  follows from Dirichlet's theorem on primes in
  arithmetic progressions, that each value
  $a \in \{ 1, 2, \ldots, m - 1 \}$ with $a$ coprime to~$m$ appears infinitely
  many times. There are $\phi (m)$ many such values of $a$, where $\phi$ is
  Euler's totient function. Consequently, the quantity $a (m - a)$ on the
  right-hand side of \eqref{eq:hyp:2f1:modp} takes $\phi (m) / 2$ many
  different values as $p$ ranges through all primes $p > m$.
  
  On the other hand, if $A_m (n)$ is a constant term sequence, then by
  \eqref{eq:ct:p:large:0:c} there exists a constant $c \in \mathbb{Q}$ such
  that $m^{2 p} A_m (p) \equiv c \pmod{p}$ for all large
  enough primes~$p$. If \eqref{eq:hyp:2f1:modp} holds for infinitely many $p$, we
  necessarily have $c = a (m - a)$, which is only possible if $\phi (m) / 2 =
  1$.
  
  Thus, if $\phi (m) > 2$ then $m^{2 p} A_m (p)$ cannot satisfy the
  congruences \eqref{eq:ct:p:large:0:c} for all large enough primes and,
  hence, the sequences $m^{2 n} A_m (n)$ and $A_m (n)$ cannot be constant
  terms. Since $\phi (m) > 2$ for all integers $m \geq 2$ except for $m
  \in \{ 2, 3, 4, 6 \}$, the claim follows.
\end{proof}

For hypergeometric sequences, we therefore have the following inclusions
\begin{equation*}
  \left\{ \text{constant terms} \right\} \subsetneq \left\{ \text{diagonals}
   \right\} \subseteq \left\{ \text{P-finite \& globally bounded seq's} \right\}.
\end{equation*}
We note that these inclusions are also true for C-finite as well as for P-finite sequences. An example for the strictness of the first inclusion in the realm of hypergeometric sequences is given by the sequence~\eqref{eq:hyp:2f1:5} and in the class of C-finite sequences by the Fibonacci numbers. The second inclusion is a consequence of a result due to Lipshitz~\cite{Lipshitz88} and it is strict if and only if Christol's \cref{conj:christol}
(restricted to hypergeometric sequences) is false.
A potential candidate of a globally bounded hypergeometric sequence that is not a diagonal is sequence \eqref{eq:christolsexample}.
We now show that this sequence is not a constant term.

\begin{lemma} \label{lem:christol_hypergeom}
  The hypergeometric sequence $A (n)$ defined in \eqref{eq:christolsexample}
  is not a constant term sequence.
\end{lemma}

\begin{proof}
  Proceeding as in the proof of \cref{lem:hyp:2f1:ct}, we find
  \begin{equation*}
    m^p \left(\tfrac{r}{m} \right)_p = r (m + r) \cdots ((p - 1) m + r),
  \end{equation*}
  where the right-hand side is a product over all residues modulo $p$. Exactly
  one factor is of the form $a p$ where $a \in \{ 1, 2, \ldots, m - 1 \}$ is
  characterized by $a p \equiv r \pmod{m}$. In that case,
  \begin{equation*}
    \frac{m^p \left(\tfrac{r}{m} \right)_p}{p!} \equiv a \pmod{p} .
  \end{equation*}
  If $p \equiv 1 \pmod{9}$, we therefore find
  \begin{equation*}
    \frac{9^p \left(\tfrac{1}{9} \right)_p}{p!} \equiv \frac{3^p \left(\tfrac{1}{3} \right)_p}{p!} \equiv 1, \quad \frac{9^p \left(\tfrac{4}{9}
     \right)_p}{p!} \equiv 4, \quad \frac{9^p \left(\tfrac{5}{9}
     \right)_p}{p!} \equiv 5 \pmod{p},
  \end{equation*}
  which combine to
  \begin{equation*}
    3^{5 p} A (p) = 3^{5 p} \frac{\left(\tfrac{1}{9} \right)_p \left(\tfrac{4}{9} \right)_p \left(\tfrac{5}{9} \right)_p}{p!^2 \left(\tfrac{1}{3} \right)_p} \equiv \frac{1 \cdot 4 \cdot 5}{1} \equiv 20 \pmod{p} .
  \end{equation*}
  On the other hand, if $p \equiv - 1 \pmod{9}$, then
  \begin{equation*}
    \frac{9^p \left(\tfrac{1}{9} \right)_p}{p!} \equiv 8, \quad \frac{3^p
     \left(\tfrac{1}{3} \right)_p}{p!} \equiv 2, \quad \frac{9^p \left(\tfrac{4}{9} \right)_p}{p!} \equiv 5, \quad \frac{9^p \left(\tfrac{5}{9}
     \right)_p}{p!} \equiv 4 \pmod{p},
  \end{equation*}
  which combine to
  \begin{equation*}
    3^{5 p} A (p) = 3^{5 p} \frac{\left(\tfrac{1}{9} \right)_p \left(\tfrac{4}{9} \right)_p \left(\tfrac{5}{9} \right)_p}{p!^2 \left(\tfrac{1}{3} \right)_p} \equiv \frac{8 \cdot 4 \cdot 5}{2} \equiv 80 \pmod{p} .
  \end{equation*}
  As in the proof of \cref{lem:hyp:2f1:ct} we conclude that $3^{5 n} A
  (n)$ and, hence, $A (n)$ cannot be a constant term.
\end{proof}

\paragraph{Acknowledgments.} Our warm thanks go to the \href{https://www-polsys.lip6.fr/}{PolSys team} of the Sorbonne University for having provided a very nice atmosphere for the three authors to collaborate during summer 2022. We are grateful to \href{https://sites.google.com/view/carlo-sanna-math}{Carlo Sanna} who provided the elegant and elementary proof of \cref{conj:new} and permitted us to include it in our article. We also thank the anonymous referee for the careful reading and for bringing \cref{conj:new} in connection to Fatou's lemma (see \cref{rem:fatou}).

The first and third authors were supported by ANR-19-CE40-0018 \href{https://specfun.inria.fr/chyzak/DeRerumNatura/}{De Rerum Natura} and the \href{https://oead.at/en/}{WTZ collaboration}/\href{https://www.campusfrance.org/}{Amadeus project} FR-09/2021 (46411YJ).
The second author gratefully acknowledges support through a Collaboration Grant (\#514645) awarded by the Simons Foundation.
The third author is funded by DOC scholarship P-26101 of the \href{https://www.oeaw.ac.at/en/1/austrian-academy-of-sciences}{ÖAW}.

\bibliographystyle{alphaabbr}

\newcommand{\etalchar}[1]{$^{#1}$}

\end{document}